\documentclass[11pt, reqno,a4paper]{article}
\usepackage{amsfonts,anysize,hyperref,amsmath,indentfirst,geometry,xcolor,cite}
\usepackage{bm}
\usepackage{amssymb}
\usepackage{epsfig}
\usepackage{graphics}
\usepackage{epsfig}
\usepackage{placeins}
\usepackage{mathrsfs}
\usepackage{amssymb}
\marginsize{25mm}{25mm}{15mm}{15mm}
\usepackage{amsthm}
\newtheorem{thm}{Theorem}[subsection]

\allowdisplaybreaks
\numberwithin{equation}{section}
\usepackage{graphicx}
\usepackage{amsmath}
\usepackage{amssymb}
\usepackage{amsfonts,amssymb,amsbsy,amsmath}
\usepackage{eqnarray}
\usepackage{diagbox}
\usepackage{mathptmx}
\usepackage{algorithm}
\usepackage{algorithmic}
\usepackage{titlesec}
\usepackage{booktabs}
\titlelabel{\thetitle.\quad}
\titlelabel{\thm.\quad}
 \newtheorem{theorem}{Theorem}[section]
\newtheorem{definition}[theorem]{Definition}
\newtheorem{corollary}[theorem]{Corollary}
\newtheorem{lemma}[theorem]{Lemma}
\newtheorem{remark}[theorem]{Remark}

\allowdisplaybreaks
 \titlelabel{\thetitle.\quad}
\begin{document}
\title{\Large \bfseries A note on condition numbers for generalized inverse $C^{\ddagger}_{A}$ and their statistical estimation\thanks
{The work is supported by the National Natural Science Foundation of China (Grant Nos. 11771265 ).}}
\author{Mahvish Samar\thanks{Corresponding author. E-mail addresses: mahvishsamar@hotmail.com}, Abdual Shahkoor}
\date{\small { Department of Mathematics, Shantou University, Shantou 515063, P. R. China \\
Department of Mathematics, Khwaja Fareed University of Engineering and Information Technology, Rahim Yar Khan 64200, Pakistan
}} \maketitle

 \maketitle{\raggedleft\bfseries{\normalsize Abstract}}\\
In this paper, we consider the condition number for the generalized inverse $C^{\ddagger}_{A}$. We first present the explicit expression of normwise mixed and componentwise condition numbers. Then, we derive the
explicit expression of normwise condition number without Kronecker product using the classical method for condition numbers. With the intermediate result, i.e., the derivative of $C^{\ddagger}_{A}$,
 we can recover the explicit expressions of condition numbers for solution of Indefinite least squares problem with equality constraint.
 To estimate these condition numbers with high reliability, we choose the probabilistic spectral norm estimator and the small-sample statistical condition estimation method and devise three algorithms. Numerical experiments are provided to illustrate
the obtained results.
  \vspace{.3cm}

{\raggedleft \em AMS classification:}\ 65F20, 65F35, 65F30, 15A12, 15A60
\vspace{.3cm}

{\raggedleft \em Keywords:} Generalized inverse $C^{\ddagger}_{A}$ ; Normwise condition number; Mixed and Componentwise condition numbers; ILSEP; Probabilistic spectral norm estimator; Small-sample statistical
condition estimation

\section{Introduction}\label{sec.1}
\vspace{-2pt}
The generalized inverse $C^{\ddagger}_{A}$ is defined by
\begin{align}\label{1.1}
C^{\ddagger}_{A}
=(I-(PQP)^{\dag}Q)C^{\dag},
\end{align}
where $Q=A^{T}JA $, $A \in \mathbb{R}^{(p+q)\times n}$ is the weight matrix, and $P=I-C^{\dag}C$ with $C^{\dag}$ denoting the Moore-Penrose inverse of $C$ and $C$ may not have full
rank and $J$ is a signature matrix defined by
\begin{equation}\label{1.2}
J = \left[ {\begin{array}{*{20}{c}}
{{I_p} }&0\\
0&{- {I_q}}
\end{array}} \right], \quad p+q=m.
\end{equation}
To make the generalized inverse $C^{\ddagger}_{A}$ be unique  (see \cite[Theorem 2.1]{[1]} or \cite[Theorem 2.2]{[2]}), throughout this paper, we assume that $Q=A^{T}JA $ is positive definite and
\begin{align}\label{1.2}
\mathrm{rank}\begin{pmatrix}
      A \\
      C \\
    \end{pmatrix}=n
.
\end{align}

 The generalized inverse $C^{\ddagger}_{A}$ is from the indefinite least squares problem with equality constraint (ILSEP).
 The ILSEP problem and its special cases have attracted many researchers to study its algorithms, error analysis, and perturbation theory (see e.g., \cite{[4],[5],[6],[7],[8],[9],[10],[11]}). Here, we only introduce some works on perturbation analysis of ILSEP.
Liu and wang \cite{[1]} first investigated the perturbation theory for this problem. The obtained results were extended by
Shi and Liu \cite{[8]}  based on the hyperbolic MGS elimination method. Later, Wang  \cite{[9]} revisited the perturbation theory of ILSEP and recovered its upper bound. In 2015, Li and Wang \cite{[12]} gave the  mixed and componentwise condition numbers of ILSEP.
It is worth to mention that the systematic
theory for normwise condition number was first given by Rice \cite{[15]}  and the terminologies of mixed and
componentwise condition numbers were first introduced by Gohberg and Koltracht \cite{[16]}.

When $q = 0$ and $C$ has full row rank, the $C^{\ddagger}_{A}$ reduce to generalized inverse $C^{\dag}_{K}$, which is from the  least squares problem with equality constraint (LSEP). Eld\'{e}n \cite{[3]} discussed generalized inverses its algorithm; Wei and Zhang \cite{[2]} extended $C^{\dag}_{K}$ to the $MK$-weighted generalized inverse $C^{\dag}_{MK}$ and discussed its structure and uniqueness; Wei \cite{[17]} studied the expression of $C^{\dag}_{K}$ based on GSVD; Gulliksson et al. \cite{[71]} presented a perturbation equation of $C^{\dag}_{K}$. Recently, Mahvish et al. \cite{[72]} introduced its condition numbers.

 However, to our best knowledge, there is no work on condition numbers of $C^{\ddagger}_{A}$ so far.  It is interesting to investigate the condition numbers because they play an important role in the research of ILSEP.
Specifically, we will discuss the normwise, mixed and componentwise condition numbers for $C^{\ddagger}_{A}$, whose explicit expressions are given in Section 3. Meanwhile, in Section 3, we also discuss how to recover the expressions of condition numbers for solution of ILSEP and how to obtain the corresponding results for residuals of this problem with the help of the derivative of $C^{\ddagger}_{A}$. Considering that it is expensive to compute these condition numbers, we investigate the statistical estimation of these condition numbers by the probabilistic spectral norm estimator \cite{[18]} and the small-sample statistical condition estimation (SSCE) method \cite{[19]}. Three related algorithms are devised in Section 4. In addition, Section 2 provides some useful notation and preliminaries and Section 5 presents some
 numerical examples to illustrate the obtained results.

\vspace{-6pt}

\section{Preliminaries}
\label{sec.2}
\vspace{-2pt}
To give the definitions of condition numbers, we first introduce the entry-wise division \cite{[20]} between the vectors $a, b \in \mathbb{R}^{m}$ defined by
\begin{eqnarray}\label{2.56}
\frac{a}{b}={\rm{diag}}(b^{\ddagger})a,
\end{eqnarray}
where ${\rm{diag}}(b^{\ddagger})$ is diagonal with diagonal elements $b^{\ddagger}_{1},...,b^{\ddagger}_{m}.$ Here,
for a number $c \in \mathbb{R}$,  $c^{\ddagger}$ is defined by
\begin{eqnarray*}
 c^{\ddagger}= \begin{cases} {c^{-1}},       & {\rm if} \ c\neq 0,  \\
1 , &{\rm if }\  c=0.
 \end{cases}
   \end{eqnarray*}
It is obvious that $\frac{a}{b}$ has components $\left(\frac{a}{b}\right)_{i}=b_{i}^{\ddagger}a_{i}.$ Similarly,
for $A=(a_{ij})\in \mathbb{R}^{m\times n},\,\, B=(b_{ij}) \in \mathbb{R}^{m\times n},$ we define $\frac{A}{B}$ as follows
    $$\left(\frac{A}{B}\right)_{ij}= b^{\ddagger}_{ij}a_{ij}.$$
With the entry-wise division,  we denote the relative distance between $a$ and $b$ as
\begin{align*}
d(a,b)=\left\|\frac{a-b}{b}\right\|_\infty=\mathop {\max }\limits_{i=1,\cdots,m}\left\{{|b^{\ddagger}_{i}|}{|a_{i}-b_{i}|}\right\}.
\end{align*}
That is, we consider the relative distance at nonzero components, while the absolute
distance at zero components.
In addition, for $\varepsilon >0$, we denote $B^{\circ}(a,\varepsilon)=\{x \in\mathbb{R}^{m}|\,\,|x_{i}-a_{i}|\leq \epsilon |a_{i}|, i=1,\cdots,m\}$ and $B(a,\varepsilon)=\{x\in\mathbb{R}^{m} \mid\|x-a\|_2\leqslant \varepsilon\|a\|_2\}$.

Now, we list the definitions of the normwise, mixed and componentwise condition numbers.
\begin{definition}\label{dfn2.1}{\rm (\cite{[20]})}\ \
Let $F:\mathbb{R}^{p}\rightarrow \mathbb{R}^{q}$ be a continuous mapping defined on an open set ${\rm {Dom}}(F)\subset \mathbb{R}^{p}$ 
and $a \in {\rm {Dom}}(F)$ satisfy $ a\neq 0$ and $F(a)\neq 0$.

{\rm (i)}\ \ The normwise condition number of $F$ at $a$ is defined by
\begin{eqnarray*}
 \qquad \qquad \qquad \qquad n(F,a) &=& \lim_{\varepsilon\rightarrow 0}\sup_{\substack{x\in B (a,\varepsilon)\\x\neq a}}\left(\frac{\|F(x)-F(a)\|_{2}}{\|F(a)\|_{2}}\Big/\frac{\|x-a\|_{2}}{\|a\|_{2}}\right).
\end{eqnarray*}

{\rm (ii)}\ \ The mixed condition number of $F$ at $a$ is defined by \
\begin{eqnarray*}
\qquad \qquad   m(F,a) &=&\lim_{\varepsilon\rightarrow 0}\sup_{\substack{x\in B^{o}(a,\varepsilon)\\x\neq a}}\frac{\|F(x)-F(a)\|_{\infty}}{\|F(a)\|_{\infty}}\frac{1}{d(x,a)}.
\end{eqnarray*}

{\rm (iii)}\ \ The componentwise condition number of $F$ at $a$ is defined by

\begin{eqnarray*}
  c(F,a) &=&  \lim_{\varepsilon\rightarrow 0}\sup_{\substack{x\in B^{o}(a,\varepsilon)\\x\neq a}}\frac{d\big(F(x),F(a)\big)}{d(x,a)}.
\end{eqnarray*}
\end{definition}
With the Fr\'echet derivative, the following lemma gives the explicit representations of these three condition numbers.
\begin{lemma}\label{lem2.3} {\rm (\cite{[20]})}
With the same assumptions as in Definition 2.1, and if $F$ is $Fr\acute{e}chet$ differentiable at $a$, we have
\begin{center}
$n(F, a)= \frac{\| \mathrm{d}F(a)\|_{2}\|a\|_{2}}{\|F(a)\|_{2}}$,\
$m(F, a)= \frac{\| |\mathrm{d}F(a)||a|\|_{\infty}}{\|F(a)\|_{\infty}}$,\
$c(F, a)= \left\|\frac{|\mathrm{d}F(a)||a|}{|F(a)|}\right\|_{\infty}$,
\end{center}
where $\mathrm{d}F(a)$ is the Fr\'echet derivative of $F$ at $a$.
\end{lemma}
The operator `vec' defined as
\begin{eqnarray*}
{\rm vec}(A)=\left[a_{1}^{T},\cdots,a_{n}^{T}\right]^{T}\in \mathbb{R}^{mn},
\end{eqnarray*}
for $A=[a_{1},\cdots,a_{n}]\in\mathbb{R}^{m\times n}$ with $a_{i}\in \mathbb{R}^{m}$
and the Kronecker product
between $A=(a_{ij})\in\mathbb{R}^{m\times n}$ and $B \in {\mathbb{R}^{p \times q}}$ defined as
 $ A\otimes B= \left[ a_{ij}B\right] \in \mathbb{R}^{mp \times nq}$
play  important roles in obtaining the expressions of the condition numbers. Some useful results on these two tools are introduced as follows \cite[Chapter 4]{[21]}.

\begin{eqnarray}
{\rm {vec}}(AXB)&\!\!=\!\!&\left(B^{T}\otimes A\right){\rm vec}(X)\label{2.2},\\
{\rm {vec}}\left(A^{T}\right)&\!\!=\!\!&\Pi_{mn} {\rm vec}(A)\label{2.3},\\
\|A\otimes B\|_{2}&\!\!= \!\!&  \|A\|_{2}\|B\|_{2}\label{2.4},\\
(A \otimes B)^T&\!\!= \!\!&  (A^T \otimes B^T),\label{2.5}\\
(A \otimes B)(C \otimes D)&\!\!= \!\!&(AC) \otimes (BD),\label{2.7}\\
\Pi_{pm} (A \otimes B) \Pi_{nq}&\!\!= \!\!& (B \otimes A),\nonumber
\end{eqnarray}
where $X$, $C$ and $D$ are of suitable orders, and $\Pi_{mn} \in {\mathbb{R}^{mn\times mn}}$ is the { vec-permutation matrix} which depends only on the dimensions $m$ and $n$. Note that if $n=1$, then $ \Pi_{nq}=I_q$ and hence
\begin{eqnarray}\label{2.6}
\Pi_{pm} (A \otimes B)&\!\!= \!\!& (B \otimes A).
\end{eqnarray}
In addition, 
from the definition of Kronecker product, we also have that when $m=1$ and $q=1$, i.e., when $A$ is a row vector and $B$ is a column vector,
\begin{eqnarray}\label{1.4}
A \otimes B=BA.
\end{eqnarray}
Besides, the following two lemmas are also useful for deriving the condition numbers and their upper bounds.
\begin{lemma}\label{lem2.5}{\rm (\cite[P. 174, Theorem 5]{[22]})}
Let $S$ be an open subset of $\mathbb{R}^{n\times q}$ , and let $F : S \longrightarrow \mathbb{R}^{m\times p}$ be a matrix function
defined and $k \geq 1$ times (continuously) differentiable on S. If $\mathrm{rank}(F(X))$ is constant on $S,$ then $F^{\dag} : S \longrightarrow \mathbb{R}^{p \times m}$ is $k$ times (continuously) differentiable
on $S$, and
\begin{eqnarray}\label{2.10}
\mathrm{d}F^{\dag}&=&-F^{\dag}\mathrm{d}FF^{\dag}+F^{\dag}F^{\dag^{T}}\mathrm{d}F^{T}(I_{m}-FF^{\dag})+(I_{p}-F^{\dag}F)\mathrm{d}F^{T}F^{\dag^{T}}F^{\dag}.
\end{eqnarray}
\end{lemma}
\begin{lemma}\label{lem2.4}{\rm (\cite{[23]})}
For any matrices $U,$ $V,$ $C,$ $D,$ $R$ and $S$ with suitable dimensions such that

$$[U\otimes V+(C\otimes D)\Pi]{\rm {vec}}(R),$$
$$\frac{[U\otimes V+(C\otimes D)\Pi]{\rm {vec}}(R)}{S},$$
$$VRU^{T}  and \  DR^{T}C^{T},$$ \\
are well-defined, we have
$$\left\||[U\otimes V+(C\otimes D)\Pi]|{\rm {vec}}(|R|)\right\|_{\infty}\leq \left\|{\rm {vec}}(|V||R||U|^{T}+|D||R|^{T}|C|^{T})\right\|_{\infty}$$
and
$$\left\|\frac{|[U\otimes V+(C\otimes D)\Pi]|{\rm {vec}}(|R|)}{|S|}\right\|_{\infty}\leq \left\|\frac{{\rm {vec}}(|V||R||U|^{T}+|D||R|^{T}|C|^{T})}{|S|}\right\|_{\infty}.$$
\end{lemma}

\vspace{-6pt}
\section{Condition numbers}
\label{sec.3}
As done in \cite{[24]}, we first give the definitions of the normwise, mixed, and componentwise condition numbers for $C^{\ddagger}_{A}$:
\begin{align}
n^{\ddagger}(A,C)&:=\lim_{\varepsilon \rightarrow 0}\sup_{\|[\Delta A , \ \Delta C]\|_{F} \leq \varepsilon \|[A ,\ C]\|_{F}}\frac{\|
(C+\Delta C)^{\ddagger}_{A}- C^{\ddagger}_{A}\|_{F}/ \|C^{\ddagger}_{A}\|_{F}}{\|[\Delta A , \ \Delta C]\|_{F}/\|[A ,\ C]\|_{F}}, \label{3.2}\\
  m^{\ddagger}(A,C)&:=\lim_{\varepsilon \rightarrow 0}\sup_{\substack{\|\Delta A /A \|_{\max}\leq \varepsilon \\\|\Delta C/C\|_{\max} \leq \varepsilon }}\frac{\|
(C+\Delta C)^{\ddagger}_{A} -C^{\ddagger}_{A}\|_{\max}}{ \|C^{\ddagger}_{A}\|_{\max}}\frac{1}{d(v+\Delta v, v)},\label{3.3}\\
  c^{\ddagger}(A,C)&:=\lim_{\varepsilon \rightarrow 0}\sup_{\substack{\|\Delta A /A\|_{\max}  \leq \varepsilon\\ \|\Delta C/C\|_{\max}\leq \varepsilon }}\frac{1}{d(v+\Delta v, v)}\left\|\frac{(C+\Delta C)^{\ddagger}_{A}- C^{\ddagger}_{A}}{C^{\ddagger}_{A}}\right\|_{\max}\label{3.4}.
\end{align}
Here, $v=({\rm {vec}}(K)^{T},{\rm {vec}}(L)^{T})^{T}$, $\Delta v=({\rm {vec}}(\Delta K)^{T},{\rm {vec}}(\Delta L)^{T})^{T}$, and for a matrix $A=(a_{ij})$,
$\|A\|_{\max}=\|{\rm vec}(A)\|_{\infty}=\mathop {\max }\limits_{i, j} |a_{ij}|.$

Using the definitions of the operator ${\rm vec}$, and the spectral, Frobenius and Max norms, we can rewrite the above definitions 
as follows:
\begin{align}
n^{\ddagger}(A,C)&:=\lim_{\varepsilon \rightarrow 0}\sup_{\left\|\left[\begin{array}{c} {\rm {vec}}(\Delta A) \\
                                                                  {\rm {vec}}( \Delta C )\\
                                                                   \end{array}
                                                                 \right] \right\|_{2}\leq \varepsilon \left\| \left[\begin{array}{c}{\rm {vec}}(A) \\
                                                                    {\rm {vec}}(C) \\
                                                                   \end{array}
                                                                 \right] \right\|_{2}}\frac{ \|{\rm {vec}}(
(C+\Delta C)^{\ddagger}_{A}- C^{\ddagger}_{A})\|_{2}}{\|{\rm {vec}}(C^{\ddagger}_{A})\|_{2}} \Big/ \small\frac{\left\|\left[\begin{array}{c} {\rm {vec}}(\Delta A) \\
                                                                  {\rm {vec}}( \Delta C )\\
                                                                   \end{array}
                                                                 \right] \right\|_{2}}{\left\| \left[\begin{array}{c}{\rm {vec}}( A) \\
                                                                    {\rm {vec}}(C) \\
                                                                   \end{array}
                                                                 \right] \right\|_{2}}, \label{361}\\
  m^{\ddagger}(A,C)&:=\lim_{\varepsilon \rightarrow 0}\sup_{\substack{\|{\rm {vec}}(\Delta A)/ {\rm {vec}}(A) \|_{\infty}\leq \varepsilon \\\|{\rm {vec}}(\Delta C) / {\rm {vec}}(C)\|_{\infty} \leq \varepsilon }}\frac{\|{\rm {vec}}(
(C+\Delta C)^{\ddagger}_{A} -C^{\ddagger}_{A})\|_{\infty}}{\|{\rm {vec}}(C^{\ddagger}_{A})\|_{\infty}}\frac{1}{d(v+\Delta v, v)},\label{362}\\
  c^{\ddagger}(A,C)&:=\lim_{\varepsilon \rightarrow 0}\sup_{\substack{\|{\rm {vec}}(\Delta A) / {\rm {vec}}(A)\|_{\infty}  \leq \varepsilon\\ \|{\rm {vec}} (\Delta C) / {\rm {vec}}(C)\|_{\infty}\leq \varepsilon }}\frac{1}{ d(v+\Delta v, v)}\left\|\frac{{\rm {vec}} ((C+\Delta C)^{\ddagger}_{A}- C^{\ddagger}_{A})}{{\rm {vec}} (C^{\ddagger}_{A})}\right\|_{\infty}\label{363}.
\end{align}
 Thus, if we 
define a mapping $\phi:\mathbb{R}^{mn+sn}\rightarrow\mathbb{R}^{ns}$ by
\begin{align}\label{3.1}
 \phi(v)={\rm{vec}}( C^{\ddagger}_{A}),
\end{align}
we have
$$n^{\ddagger}(A,C)=n(\phi, v),\ m^{\ddagger}(A,C)=m(\phi, v),\ c^{\ddagger}(A,C)=c(\phi, v).$$
In the following, we find the expression of Fr\'{e}chet derivative of $\phi$ at $v$.
\begin{lemma}\label{lem3.1}
 The mapping $\phi$ is continuous, Fr$\acute{e}$chet differentiable at $v$, and
 $$\mathrm{d}\phi(v)=[W(A), W(C)],$$
 where
 \begin{align*}
  W(A)&= -[(C^{\ddagger^{T}}_{A} \otimes (PQP)^{\dag}A^{T}J)+((JAC^{\ddagger}_{A})^{T}\otimes (PQP)^{\dag})\Pi_{mn}],\\
 W(C)&= -[(C^{\ddagger^{T}}_{A}\otimes C^{\ddagger}_{A} )-((I-CC^{\dag})^{T}\otimes C^{\ddagger}_{A}C^{\dag^{T}})\Pi_{sn}-(C^{\dag^{T}}QC^{\ddagger}_{A})^{T}\otimes (PQP)^{\dag})\Pi_{sn}],
 \end{align*}

\end{lemma}
\begin{proof}
Differentiating both sides of \eqref{1.1}, we get
\begin{eqnarray}\label{5001}
\mathrm{d}(C^{\ddagger}_{A})&=& \mathrm{d}[(I-(PQP)^{\dag}Q)C^{\dag}].
\end{eqnarray}
From \cite[Lemma 2.3]{[1]}, we have
\begin{align}\label{3001}
&(PQP)^{\dag}=P(PQP)^{\dag}=(PQP)^{\dag}P=P(PQP)^{\dag}P,\\
&P(I-(PQP)^{\dag}QP)=0,\,\ (PQP)^{\dag}QP=P.\label{3002}
\end{align}
Thus, substituting \eqref{3001} into \eqref{5001} and using some basic knowledge on differential give
\begin{align*}
\mathrm{d}(C^{\ddagger}_{A})&=\mathrm{d}[(I-P(PQP)^{\dag}Q)C^{\dag}]=\mathrm{d}C^{\dag}- \mathrm{d}(P(PQP)^{\dag}QC^{\dag})\\
&=(I-P(PQP)^{\dag}Q)\mathrm{d}C^{\dag}-\mathrm{d}P(PQP)^{\dag}QC^{\dag}-P\mathrm{d}(PQP)^{\dag}QC^{\dag}-P(PQP)^{\dag}\mathrm{d}QC^{\dag}.
\end{align*}
Further, using \eqref{2.10},  we have
\begin{align*}
\mathrm{d}(C^{\ddagger}_{A})&=(I-P(PQP)^{\dag}Q)[-C^{\dag}\mathrm{d}CC^{\dag}+C^{\dag}C^{\dag^{T}}\mathrm{d}C^{T}(I-CC^{\dag})+(I-C^{\dag}C)\mathrm{d}C^{T}C^{\dag^{T}}C^{\dag}]\nonumber\\
&\quad-\mathrm{d}(I-C^{\dag}C)(PQP)^{\dag}QC^{\dag}+P[(PQP)^{\dag}\mathrm{d}(PQP)(PQP)^{\dag}\nonumber\\
&\quad-(PQP)^{\dag}(PQP)^{\dag^{T}}\mathrm{d}(PQP)^{T}(I-(PQP)(PQP)^{\dag})\nonumber\\
&\quad-(I-(PQP)^{\dag}(PQP))\mathrm{d}(PQP)^{T}(PQP)^{\dag^{T}}(PQP)^{\dag}]QC^{\dag}-P(PQP)^{\dag}\mathrm{d}QC^{\dag}. 
\end{align*}
Noting  \eqref{3001}, \eqref{1.1}, and the result $(I-P(PQP)^{\dag}Q)(I-C^{\dag}C)=P(I-(PQP)^{\dag}QP)$, the above equation can be rewritten as
\begin{align*}
\mathrm{d}(C^{\ddagger}_{A})&=-C^{\ddagger}_{A}\mathrm{d}CC^{\dag}+C^{\ddagger}_{A}C^{\dag^{T}}\mathrm{d}C^{T}(I-CC^{\dag})
+P(I-(PQP)^{\dag}QP)\mathrm{d}C^{T}C^{\dag^{T}}C^{\dag}+\mathrm{d}C^{\dag}C(PQP)^{\dag}QC^{\dag}\nonumber\\
&\quad+C^{\dag}\mathrm{d}C(PQP)^{\dag}QC^{\dag}+(PQP)^{\dag}\mathrm{d}(PQP)(PQP)^{\dag}QC^{\dag}\nonumber\\
&\quad-(PQP)^{\dag}(PQP)^{\dag^{T}}\mathrm{d}(PQP)^{T}(I-(PQP)(PQP)^{\dag})QC^{\dag}\nonumber\\
&\quad -P(I-(PQP)^{\dag}QP)\mathrm{d}(PQP)^{T}(PQP)^{\dag^{T}}(PQP)^{\dag}QC^{\dag}-(PQP)^{\dag}\mathrm{d}QC^{\dag}.\nonumber
\end{align*}
Further, by the fact $PQ=(QP)^{T}=QP$, \eqref{3002}, and
\begin{align}\label{3003}
 CP(PQP)^{\dag}&= C(PQP)^{\dag}=0,
\end{align}
we can simplify the above equation as
\begin{align}
\mathrm{d}(C^{\ddagger}_{A})&=-C^{\ddagger}_{K}\mathrm{d}CC^{\dag}+C^{\ddagger}_{A}C^{\dag^{T}}\mathrm{d}C^{T}(I-CC^{\dag})+C^{\dag}\mathrm{d}C(PQP)^{\dag}QC^{\dag}
-(PQP)^{\dag}\mathrm{d}QC^{\dag}\nonumber\\
&\quad
+(PQP)^{\dag}\mathrm{d}QP(PQP)^{\dag}QC^{\dag}+(PQP)^{\dag}Q\mathrm{d}P(PQP)^{\dag}QC^{\dag}\nonumber\\
&\quad-(PQP)^{\dag}(PQP)^{\dag^{T}}\mathrm{d}Q^{T}P^{T}(I-(PQP)(PQP)^{\dag})QC^{\dag}\nonumber\\
&\quad-(PQP)^{\dag}(PQP)^{\dag^{T}}Q^{T}\mathrm{d}P^{T}(I-(PQP)(PQP)^{\dag})QC^{\dag}. \\
&=-C^{\ddagger}_{K}\mathrm{d}CC^{\dag}+C^{\ddagger}_{A}C^{\dag^{T}}\mathrm{d}C^{T}(I-CC^{\dag})+C^{\dag}\mathrm{d}C(PQP)^{\dag}QC^{\dag}
-(PQP)^{\dag}\mathrm{d}A^{T}JAC^{\dag}\nonumber\\
&\quad
-(PQP)^{\dag}A^{T}J\mathrm{d}AC^{\dag}+(PQP)^{\dag}P\mathrm{d}A^{T}JA(PQP)^{\dag}QC^{\dag}+(PQP)^{\dag}PA^{T}J\mathrm{d}A(PQP)^{\dag}QC^{\dag}\nonumber\\
&\quad-(PQP)^{\dag}(PQP)^{\dag^{T}}\mathrm{d}Q^{T}P(I-QP(PQP)^{\dag})QC^{\dag}+(PQP)^{\dag}Q\mathrm{d}P(PQP)^{\dag}QC^{\dag}\nonumber\\
&\quad-(PQP)^{\dag}(PQP)(PQP)^{\dag}\mathrm{d}P^{T}(I-QP(PQP)^{\dag})QC^{\dag}.  \label{3004}
\end{align}
Considering $PQ=(QP)^{T}=QP$  we get
\begin{align}
P((I- QP(PQP)^{\dag})=0 \,\,\,\,\,\         QP(PQP)^{\dag}=P \label{5003}
\end{align}

Substituting this fact into \eqref{3004}  implies
\begin{align*}
\mathrm{d}(C^{\ddagger}_{A})&=-C^{\ddagger}_{A}\mathrm{d}CC^{\dag}+C^{\ddagger}_{A}C^{\dag^{T}}\mathrm{d}C^{T}(I-CC^{\dag})
+C^{\dag}\mathrm{d}C(PQP)^{\dag}QC^{\dag}\nonumber\\
&\quad-(PQP)^{\dag}A^{T}J\mathrm{d}A(I-P(PQP)^{\dag}Q)C^{\dag}-(PQP)^{\dag}QC^{\dag}\mathrm{d}C(PQP)^{\dag}QC^{\dag}\nonumber\\
&\quad-(PQP)^{\dag}\mathrm{d}A^{T}JA(I-(PQP)^{\dag}Q)C^{\dag}+(PQP)^{\dag}\mathrm{d}C^{T}C^{\dag^{T}}Q(I-(PQP)^{\dag}Q)C^{\dag}.
\end{align*}
Using \eqref{1.1} and \eqref{3001} again, we can rewrite the above equation as
\begin{align}
\mathrm{d}(C^{\ddagger}_{A})&=-C^{\ddagger}_{A}\mathrm{d}CC^{\ddagger}_{A}+C^{\ddagger}_{A}C^{\dag^{T}}\mathrm{d}C^{T}(I-CC^{\dag})
+(PQP)^{\dag}\mathrm{d}C^{T}C^{\dag^{T}}QC^{\ddagger}_{A}-(PQP)^{\dag} A^{T}J \mathrm{d}AC^{\ddagger}_{A}\nonumber\\
&\quad- (PQP)^{\dag}\mathrm{d}A^{T}JAC^{\ddagger}_{A}.\label{3.5}
\end{align}
Applying the `vec' operation on the both sides of \eqref{3.5}, and noting \eqref{2.2} and \eqref{2.3}, we obtain
\begin{align*}
{\rm {vec}}(\mathrm{d}(C^{\ddagger}_{A}))&= -(C^{\ddagger^{T}}_{A}\otimes (PQP)^{\dag}A^{T}J){\rm {vec}}(\mathrm{d}A)-((JAC^{\ddagger}_{A})^{T}\otimes (PQP)^{\dag}){\rm {vec}}(\mathrm{d}A^{T}) \\
&\quad-(C^{\ddagger^{T}}_{A}\otimes C^{\ddagger}_{A} ){\rm {vec}}(\mathrm{d}C)+((I-CC^{\dag})^{T}\otimes C^{\ddagger}_{A}C^{\dag^{T}}){\rm{vec}}(\mathrm{d}C^{T})\nonumber\\
&\quad+((C^{\dag^{T}}QC^{\ddagger}_{A})^{T}\otimes (PQP)^{\dag}){\rm {vec}}(\mathrm{d}C^{T}) \quad \textrm{ by \eqref{2.2}}\\
&=-[(C^{\ddagger^{T}}_{A} \otimes (PQP)^{\dag}A^{T}J)+((JAC^{\ddagger}_{A})^{T}\otimes (PQP)^{\dag})\Pi_{mn}]{\rm {vec}}(\mathrm{d}A)\\
&\quad-[(C^{\ddagger^{T}}_{A}\otimes C^{\ddagger}_{A} )-((I-CC^{\dag})^{T}\otimes C^{\ddagger}_{A}C^{\dag^{T}})\Pi_{sn}-(C^{\dag^{T}}QC^{\ddagger}_{A})^{T}\otimes (PQP)^{\dag})\Pi_{sn}]{\rm {vec}}(\mathrm{d}C) \\
&\quad \quad \textrm{ by \eqref{2.3}} \\
&=[-(C^{\ddagger^{T}}_{A} \otimes (PQP)^{\dag}A^{T}J)-((JAC^{\ddagger}_{A})^{T}\otimes (PQP)^{\dag})\Pi_{mn}, \\
&\quad -(C^{\ddagger^{T}}_{A}\otimes C^{\ddagger}_{A} )+((I-CC^{\dag})^{T}\otimes C^{\ddagger}_{A}C^{\dag^{T}})\Pi_{sn}+(C^{\dag^{T}}QC^{\ddagger}_{A})^{T}\otimes (PQP)^{\dag})\Pi_{sn}]\begin{bmatrix}
                                                              {\rm {vec}}(\mathrm{d}A) \\
                                                              {\rm {vec}}(\mathrm{d}C) \\
                                                            \end{bmatrix}.
\end{align*}
That is,
$$\mathrm{d}({\rm {vec}}(C^{\ddagger}_{A}))=[W(A),W(C)]\mathrm{d}v.$$
From the definition of Fr\'{e}chet derivative, we have the desired results.
\end{proof}

\begin{remark}\label{rem46}
{\rm The results in Lemma \ref{lem3.1} are for the case that $C$ may not have full rank. If we suppose $C$ is  full row rank  and $q=0$ than we have
$C^{\ddagger}_{A}=C^{\dag}_{A}$ and
\begin{align*}
 \widetilde{W}(A)&= -[(C^{\dag^{T}}_{A} \otimes (AP)^{\dag})+((AC^{\dag}_{A})^{T}\otimes (AP)^{\dag}(AP)^{\dag^{T}})\Pi_{mn}],\\
 \widetilde{W}(C)&= -[(C^{\dag^{T}}_{A}\otimes C^{\dag}_{A} )-((AC^{\dag})^{T}AC^{\dag}_{A})^{T}\otimes (AP)^{\dag}(AP)^{\dag^{T}})\Pi_{sn}].
 \end{align*}
 which is the result in \cite[Lemma 3.1]{[72]} with which we can recover the condition numbers for A-weighted pseudoinverse  \cite{[72]}.
  }
\end{remark}

\begin{remark}\label{rem477}
{\rm Setting $A$ to be the zero matrix and $C$ is  full column rank, we have $C^{\ddagger}_{0}=C^{\dag}$ and
\begin{eqnarray}\label{3.6}
W(A)=0,\ W(C)=[-((C^{\dag})^{T}\otimes C^{\dag})+ ((I-CC^{\dag})\otimes(C^{T}C)^{-1})\Pi_{sn}],
\end{eqnarray}
which is the result in \cite[Lemma 4]{[23]} with which we can recover the condition numbers for Moore-Penrose inverse \cite{[23]}. 
}
\end{remark}

In the following, we present the expressions for normwise, mixed and componentwise condition numbers for $C^{\ddagger}_{A},$ which are the immediate results of  Lemma \ref{lem2.3} and Lemma \ref{lem3.1}.
\begin{theorem}\label{thm3.2}
The normwise, mixed and componentwise condition numbers for $C^{\ddagger}_{A}$ defined in \eqref{3.2}-\eqref{3.4} are
\begin{eqnarray}
  n^{\ddagger} (A,C)\!\!\!&=&\!\!\! \frac{\|[W(A),W(C)]\|_{2}\left\|\begin{bmatrix}
{\rm {vec}}(K)\\
 {\rm {vec}}(L)
\end{bmatrix}\right\|_{2} }{\|{\rm {vec}}(C^{\ddagger}_{A})\|_{2}},\label{311}\\
  m^{\ddagger} (A,C) \!\!\!&=&\!\!\!\frac{\left\||W(A)|{\rm {vec}}(|A|)+|W(C)|{\rm {vec}}(|C|)\right\|_{\infty}}{\|{\rm {vec}}(C^{\ddagger}_{A})\|_{\infty}},\label{312}\\
  c^{\ddagger} (A,C) \!\!\!&=&\!\!\!\left\|\frac{|W(A)|{\rm {vec}}(|A|)+|W(C)|{\rm {vec}}(|C|)}{{\rm {vec}}(C^{\ddagger}_{A})}\right\|_{\infty}\label{313}.
\end{eqnarray}
\end{theorem}

To reduce the cost for computing these condition numbers, we
provide easier computable upper bounds.  The numerical experiments in Section \ref{sec.7} show that these bounds are often very good.
\begin{corollary}\label{col33}
 The normwise, mixed and componentwise condition numbers for $C^{\ddagger}_{A}$ can be bounded as
\begin{align*}
n^{\ddagger}(A,C)&\leq n^{upper}(A,C)\\
&=[\|C^{\ddagger}_{A}\|_{2}\|(PQP)^{\dag}A^{T}J\|_{2}+\|JAC^{\ddagger}_{A}\|_{2}\|(PQP)^{\dag}\|_{2}
+\|C^{\ddagger}_{A}\|_{2}\|C^{\ddagger}_{A}\|_{2}+\|(I-CC^{\dag})\|_{2}\|C^{\ddagger}_{A}C^{\dag^{T}}\|_{2}\\
&\quad+\| C^{\dag^{T}}QC^{\ddagger}_{A}\|_{2}\| (PQP)^{\dag} \|_{2}] \frac{\|[A,\ C]\|_{F}}{\|C^{\ddagger}_{A}\|_{F}},\\
m^{\ddagger} (A,C) &\leq m^{upper}(A,C)\\
&=\|[|(PQP)^{\dag}A^{T}J||A||C^{\ddagger^{T}}_{A}|+|(PQP)^{\dag}||A^{T}||JAC^{\ddagger}_{A}|
+|C^{\ddagger}_{A}||C||C^{\ddagger^{T}}_{A}|+|C^{\ddagger}_{A}C^{\dag^{T}}||C^{T}||(I-CC^{\dag})|\\
&\quad+|(PQP)^{\dag}||C^{T}||C^{\dag^{T}}QC^{\ddagger}_{A}|]\|_{\max} /  \|C^{\ddagger}_{A}\|_{\max}, \\
c^{\ddagger} (A,C) &\leq c^{upper}(A,C)\\
&=\|[|(PQP)^{\dag}A^{T}J||A||C^{\ddagger^{T}}_{A}|+|(PQP)^{\dag}||A^{T}||JAC^{\ddagger}_{A}|
+|C^{\ddagger}_{A}||C||C^{\ddagger^{T}}_{A}|+|C^{\ddagger}_{A}C^{\dag^{T}}||C^{T}||(I-CC^{\dag})|\\
&\quad+|(PQP)^{\dag}||C^{T}||C^{\dag^{T}}QC^{\ddagger}_{A}|]/ C^{\ddagger}_{A}\|_{\max}.
\end{align*}
\end{corollary}
\begin{proof}
Firstly, using the property on the spectral norm that for the matrices $C$ and $D$ of suitable orders, $\|[C,\ D]\|_{2} \leq \|C\|_{2} + \|D\|_{2}$, Theorem \ref{thm3.2}, and \eqref{2.4}, we have
\begin{align*}
n^{\ddagger}(A,C)
&\leq [\|-(C^{\ddagger^{T}}_{A}\otimes (PQP)^{\dag}A^{T}J)-((JAC^{\ddagger}_{A})^{T}\otimes (PQP)^{\dag})\Pi_{mn}\|_{2}\\
&\quad+\|-(C^{\ddagger^{T}}_{A}\otimes C^{\ddagger}_{A} )+((I-CC^{\dag})^{T}\otimes C^{\ddagger}_{A}C^{\dag^{T}})\Pi_{sn}+(C^{\dag^{T}}QC^{\ddagger}_{A})^{T}\otimes (PQP)^{\dag})\Pi_{sn}\|_{2}]\\
&\quad \times \frac{\|[A,\ C]\|_{F}}{\|C^{\ddagger}_{A}\|_{F}}\\
&\leq[\|C^{\ddagger^{T}}_{A} \otimes (PQP)^{\dag}A^{T}J\|_{2}+\|(JAC^{\ddagger}_{A})^{T}\otimes (PQP)^{\dag}\|_{2}+\|C^{\ddagger^{T}}_{A}\otimes C^{\ddagger}_{A} \|_{2}\\&\quad+\|(I-CC^{\dag})^{T}\otimes C^{\ddagger}_{A}C^{\dag^{T}}\|_{2}+\|(C^{\dag^{T}}QC^{\ddagger}_{A})^{T}\otimes (PQP)^{\dag} \|_{2}]\frac{\|[A,\ C]\|_{F}}{\|C^{\ddagger}_{A}\|_{F}}\\
&=[\|C^{\ddagger}_{A}\|_{2}\|(PQP)^{\dag}A^{T}J\|_{2}+\|JAC^{\ddagger}_{A}\|_{2}\|(PQP)^{\dag}\|_{2}
+\|C^{\ddagger}_{A}\|_{2}\|C^{\ddagger}_{A}\|_{2}+\|(I-CC^{\dag})\|_{2}\|C^{\ddagger}_{A}C^{\dag^{T}}\|_{2}\\
&\quad+\| C^{\dag^{T}}QC^{\ddagger}_{A}\|_{2}\| (PQP)^{\dag} \|_{2}] \frac{\|[A,\ C]\|_{F}}{\|C^{\ddagger}_{A}\|_{F}}.
\end{align*}

Secondly, by Theorem \ref{thm3.2} and Lemma \ref{lem2.4}, we obtain
\begin{align*}
 m^{\ddagger}(A, C)&=\||-(C^{\ddagger^{T}}_{A}\otimes (PQP)^{\dag}A^{T}J)-((JAC^{\ddagger}_{A})^{T}\otimes (PQP)^{\dag})\Pi_{mn}|{\rm {vec}}(|A|)
 +|-(C^{\ddagger^{T}}_{A}\otimes C^{\ddagger}_{A} )\\
&\quad +((I-CC^{\dag})^{T}\otimes C^{\ddagger}_{A}C^{\dag^{T}})\Pi_{sn}+(C^{\dag^{T}}QC^{\ddagger}_{A})^{T}\otimes (PQP)^{\dag})\Pi_{sn}|{\rm {vec}}(|C|)\|_{\infty}/\|{\rm {vec}}(C^{\ddagger}_{A})\|_{\infty}\\
&\leq \|[|(PQP)^{\dag}A^{T}J||A||C^{\ddagger^{T}}_{A}|+|(PQP)^{\dag}||A^{T}||JAC^{\ddagger}_{A}|
+|C^{\ddagger}_{A}||C||C^{\ddagger^{T}}_{A}|+|C^{\ddagger}_{A}C^{\dag^{T}}||C^{T}||(I-CC^{\dag})|\\
&\quad+|(PQP)^{\dag}||C^{T}||C^{\dag^{T}}QC^{\ddagger}_{A}|]\|_{\max} /  \|C^{\ddagger}_{A}\|_{\max},
\end{align*}
and
\begin{align*}
 c^{\ddagger}(A, C)&=\||-(C^{\ddagger^{T}}_{A}\otimes (PQP)^{\dag}A^{T}J)-((JAC^{\ddagger}_{A})^{T}\otimes (PQP)^{\dag})\Pi_{mn}|{\rm {vec}}(|A|)
 +|-(C^{\ddagger^{T}}_{A}\otimes C^{\ddagger}_{A} )\\
&\quad +((I-CC^{\dag})^{T}\otimes C^{\ddagger}_{A}C^{\dag^{T}})\Pi_{sn}+(C^{\dag^{T}}QC^{\ddagger}_{A})^{T}\otimes (PQP)^{\dag})\Pi_{sn}|{\rm {vec}}(|C|)/|{\rm {vec}}(C^{\ddagger}_{A})|\|_{\infty}\\
&\leq\|[|(PQP)^{\dag}A^{T}J||A||C^{\ddagger^{T}}_{A}|+|(PQP)^{\dag}||A^{T}||JAC^{\ddagger}_{A}|
+|C^{\ddagger}_{A}||C||C^{\ddagger^{T}}_{A}|+|C^{\ddagger}_{A}C^{\dag^{T}}||C^{T}||(I-CC^{\dag})|\\
&\quad+|(PQP)^{\dag}||C^{T}||C^{\dag^{T}}QC^{\ddagger}_{A}|]/ C^{\ddagger}_{A}\|_{\max}.
\end{align*}
\end{proof}

For normwise condition number, we have an alternative form, which doesn't contain the Kronecker product.
\begin{theorem}\label{thm3.6}  The normwise condition number $n^{\ddagger}(A,C)$ for $C^{\ddagger}_{A}$ has the following
equivalent form
\begin{align}\label{31}
n^{\ddagger}(A,C)&= \frac{\|V\|^{1/2}_{2} \left\|\begin{bmatrix}
{\rm {vec}}(A)\\
 {\rm {vec}}(C)
\end{bmatrix}\right\|_{2} }{\|{\rm {vec}}(C^{\ddagger}_{A})\|_{2}},
\end{align}
where
\begin{align}\label{91}
V&=\left(\|JAC^{\ddagger}_{A}\|^{2}_{2}+\|C^{\dag^{T}}QC^{\ddagger}_{A}\|^{2}_{2}\right)  (PQP)^{\dag^{2}}+\|C^{\ddagger}_{A}\|^{2}_{2}
\left(\|(PQP)^{\dag}A^{T}J\|^{2}_{2}+ C^{\ddagger}_{A}C^{\ddagger^{T}}_{A}  \right)+\|I-CC^{\dag}\|^{2}_{2}\|C^{\ddagger}_{A}C^{\dag^{T}}\|^{2}_{2}.\nonumber\\
\end{align}
\end{theorem}
\begin{proof}
Note that
\begin{align}\label{320}
\|[V(A),V(C)]\|_{2}=\|[V_1,V_2]\|_{2}=\|[V_1,V_2][V_1,V_2]^{T}\|^{1/2}_{2}=\|V_{1}V^{T}_{1}+V_{2}V^{T}_{2}\|^{1/2}_{2}.
\end{align}
In the following, we compute $V_{1}V^{T}_{1}$ and $V_{2}V^{T}_{2}$, respectively.

Firstly, let

   $$V_{11}=-((JAC^{\ddagger}_{A})^{T}\otimes (PQP)^{\dag})\Pi_{mn},\,\,\,\,\,\,\ V_{12}=-(C^{\ddagger^{T}}_{A} \otimes (PQP)^{\dag}A^{T}J).$$
Then
\begin{align}\label{32}
V_{1}V^{T}_{1}&=V_{11}V^{T}_{11}+ V_{12}V^{T}_{12}+V_{11}V^{T}_{12}+V_{12}V^{T}_{11}.
\end{align}
By \eqref{2.5} and \eqref{2.7}, we get
\begin{align}
V_{11}V^{T}_{11}&=((JAC^{\ddagger}_{A})^{T}\otimes (PQP)^{\dag})(JAC^{\ddagger}_{A}\otimes (PQP)^{\dag^{T}})=\|JAC^{\ddagger}_{A}\|^{2}_{2}(PQP)^{\dag^{2}}, \label{33} \\
V_{12}V^{T}_{12}&=(C^{\ddagger^{T}}_{A} \otimes (PQP)^{\dag}A^{T}J)(C^{\ddagger}_{A} \otimes ((PQP)^{\dag}A^{T}J)^{T})=\|C^{\ddagger}_{A}\|^{2}_{2}
\|(PQP)^{\dag}A^{T}J\|^{2}_{2} \label{34}.
\end{align}
Note that
\begin{align}\label{35}
(PQP)^{\dag} C^{\ddagger}_{A}&=(PQP)^{\dag}(I- (PQP)^{\dag}Q)C^{\dag}\nonumber \quad \textrm{ by \eqref{1.1} }  \\
&=((PQP)^{\dag}-(PQP)^{\dag}(PQP)(PQP)^{\dag})C^{\dag}=0.
\end{align}
Thus by \eqref{2.5}, \eqref{2.7}, \eqref{2.6} and \eqref{35}, we have
\begin{align}\label{36}
V_{11}V^{T}_{12}&=((JAC^{\ddagger}_{A})^{T}\otimes (PQP)^{\dag})(((PQP)^{\dag}A^{T}J)^{T} \otimes C^{\ddagger}_{A}) \nonumber \quad \textrm{ by \eqref{2.5} and \eqref{2.6} }  \\
&=((JAC^{\ddagger}_{A})^{T}((PQP)^{\dag}A^{T}J)^{T})\otimes ((PQP)^{\dag} C^{\ddagger}_{A}) \nonumber  \quad \textrm{ by \eqref{2.7}}\\
&=0=(V_{12}V^{T}_{11}))^{T}.\quad \textrm{ by \eqref{35}}
\end{align}
Substituting \eqref{33},\eqref{34}, \eqref{36} into \eqref{32}, we get
\begin{align}\label{37}
V_{1}V^{T}_{1}&=\|JAC^{\ddagger}_{A}\|^{2}_{2}(PQP)^{\dag^{2}}+\|C^{\ddagger^{T}}_{A}\|^{2}_{2}
\|(PQP)^{\dag}A^{T}J\|^{2}_{2}.
\end{align}
Now, let\\
 $V_{21}=((I-CC^{\dag})^{T}\otimes C^{\ddagger}_{A}C^{\dag^{T}})\Pi_{sn},\,\,\,\,\,\,\ V_{22}=((C^{\dag^{T}}QC^{\ddagger}_{A})^{T}\otimes (PQP)^{\dag})\Pi_{sn},\,\,\,\,\,\,\ V_{23}=-(C^{\ddagger^{T}}_{A}\otimes C^{\ddagger}_{A}).$ \\
 Then
\begin{align}\label{38}
V_{2}V^{T}_{2}=V_{21}V^{T}_{21}+V_{22}V^{T}_{22}+V_{23}V^{T}_{23}+V_{21}V^{T}_{22}-V_{22}V^{T}_{23}-V_{21}V^{T}_{23}.
\end{align}
By \eqref{2.5} and \eqref{2.7}, we get
\begin{align}
V_{21}V^{T}_{21}&=((I-CC^{\dag})^{T}\otimes C^{\ddagger}_{A}C^{\dag^{T}})((I-CC^{\dag})\otimes (C^{\ddagger}_{A}C^{\dag^{T}})^{T})=\|I-CC^{\dag}\|^{2}_{2}\|C^{\ddagger}_{A}C^{\dag^{T}}\|^{2}_{2}. \label{4011}\\
V_{22}V^{T}_{22}&=((C^{\dag^{T}}QC^{\ddagger}_{A})^{T}\otimes (PQP)^{\dag})((C^{\dag^{T}}QC^{\ddagger}_{A})\otimes (PQP)^{\dag^{T}})=\|C^{\dag^{T}}QC^{\ddagger}_{A}\|^{2}_{2}(PQP)^{\dag^{2}} ,\label{39} \\
V_{23}V^{T}_{23}&=(C^{\ddagger^{T}}_{A}\otimes C^{\ddagger}_{A})(C^{\ddagger}_{A} \otimes C^{\ddagger^{T}}_{A})=\|C^{\ddagger}_{A}\|^{2}_{2}C^{\ddagger}_{A}C^{\ddagger^{T}}_{A}. \label{40}
\end{align}
By \eqref{2.5}, \eqref{2.7}, \eqref{2.6} and \eqref{1.4}, we obtain

\begin{align}
V_{21}V^{T}_{22}&=((I-CC^{\dag})^{T}\otimes C^{\ddagger}_{A}C^{\dag^{T}})((C^{\dag^{T}}QC^{\ddagger}_{A})\otimes (PQP)^{\dag^{T}})  \nonumber \quad \textrm{ by \eqref{2.5} and \eqref{2.6} }\\
&= ((I-CC^{\dag})^{T}C^{\dag^{T}}QC^{\ddagger}_{A})\otimes  (C^{\ddagger}_{A}C^{\dag^{T}}(PQP)^{\dag^{T}})   \nonumber  \quad \textrm{ by \eqref{2.7}}\\
&= (C^{\dag}(I-CC^{\dag}))^{T}QC^{\ddagger}_{A} \otimes  (C^{\ddagger}_{A}C^{\dag^{T}}(PQP)^{\dag^{T}})      \nonumber  \quad \textrm{ by \eqref{1.4}} \\
&= 0 \otimes (C^{\ddagger}_{A}C^{\dag^{T}}(PQP)^{\dag^{T}}) =0=(V_{22}V^{T}_{21})^{T}. \label{41}
\end{align}

\begin{align}
V_{21}V^{T}_{23}&=((I-CC^{\dag})^{T}\otimes C^{\ddagger}_{A}C^{\dag^{T}})(C^{\ddagger^{T}}_{A}\otimes C^{\ddagger}_{A})\nonumber \quad \textrm{ by \eqref{2.5} and \eqref{2.6} }\\
&= ((I-CC^{\dag})^{T}C^{\ddagger^{T}}_{A} )\otimes  (C^{\ddagger}_{A}C^{\dag^{T}} C^{\ddagger}_{A})   \nonumber  \quad \textrm{ by \eqref{2.7}}\\
&= (C^{\ddagger}_{A}(I-CC^{\dag}))^{T} \otimes  (C^{\ddagger}_{A}C^{\dag^{T}} C^{\ddagger}_{A})     \nonumber  \quad \textrm{ by \eqref{1.4}} \\
&= ((I- (PQP)^{\dag}Q)C^{\dag}(I-CC^{\dag}))^{T}  \otimes  (C^{\ddagger}_{A}C^{\dag^{T}} C^{\ddagger}_{A})=0=(V_{23}V^{T}_{21})^{T}. \label{41}
\end{align}
\begin{align}
V_{22}V^{T}_{23}&=((C^{\dag^{T}}QC^{\ddagger}_{A})^{T}\otimes (PQP)^{\dag})(C^{\ddagger^{T}}_{A}\otimes C^{\ddagger}_{A})\nonumber \quad \textrm{ by \eqref{2.5} and \eqref{2.6} }\\
&= ((C^{\dag^{T}}QC^{\ddagger}_{A})^{T}C^{\ddagger^{T}}_{A} )\otimes ( (PQP)^{\dag} C^{\ddagger}_{A})   \nonumber  \quad \textrm{ by \eqref{2.7}}\\
&= ((C^{\dag^{T}}QC^{\ddagger}_{A})^{T}C^{\ddagger^{T}}_{A} )\otimes 0=0=(V_{23}V^{T}_{22})^{T}    \nonumber  \quad \textrm{ by \eqref{1.4}} \\
\end{align}

Substituting \eqref{39}, \eqref{40}, and \eqref{41} into \eqref{38} implies
\begin{align}\label{42}
V_{2}V^{T}_{2}&=\|I-CC^{\dag}\|^{2}_{2}\|C^{\ddagger}_{A}C^{\dag^{T}}\|^{2}_{2}+\|C^{\dag^{T}}QC^{\ddagger}_{A}\|^{2}_{2}(PQP)^{\dag^{2}}  +\|C^{\ddagger}_{A}\|^{2}_{2}C^{\ddagger}_{A}C^{\ddagger^{T}}_{A}.
\end{align}

Putting  \eqref{311}, \eqref{320}, \eqref{37}, and \eqref{42} together and setting $V=V_{1}V^{T}_{1}+V_{2}V^{T}_{2}$, we have the desired result  \eqref{31}.
\end{proof}
\begin{remark}\label{rem50}
{\rm Using the GHQR factorization \cite{[9]} of the matrix pair $A$ and $C$   in \eqref{1.1} and \eqref{1.2}:
\begin{align}\label{2.8}
H^{T}AQ=&\begin{pmatrix}
           L_{11} & 0 \\
           L_{21} & L_{22} \\
         \end{pmatrix}
, \,\,\,\,\,\,\ U^{T}CQ=\begin{pmatrix}
           K_{11} & 0 \\
           0 & 0 \\
         \end{pmatrix},
\end{align}
where $U \in \mathbb{R}^{l \times l}$ and $Q \in \mathbb{R}^{n \times n}$ and a $J-$orthogonal matrix, $H \in \mathbb{R}^{(p+q)\times (p+q)}$ (i.e., $HJH^{T}=J$), $L_{22}$  and $K_{11}$ are lower triangular and non-singular. We have

$$C^{\ddagger}_{A}=Q\left(
                 \begin{array}{c}
                   I \\
                   - L^{-1}_{22} L_{21} \\
                 \end{array}
               \right)K^{-1}_{11}U^{T}_{1} ,\,\,\,\,\ (PQP)^{\dag}A^{T}J=Q\left(
                 \begin{array}{c}
                    0 \\
                    L^{-1}_{22} \\
                 \end{array}
               \right)H_{2}^{T}, \,\,\,\,\ (PQP)^{\dag}=Q\left(
                 \begin{array}{c}
                    0 \\
                   -( L^{T}_{22} L_{22})^{-1} \\
                 \end{array}
               \right)Q^{T},$$ \,\,\,\,\ $$C^{\ddagger}_{A}C^{\dag^{T}}=\left(
                 \begin{array}{c}
                    0 \\
                   - L^{-1}_{22} L_{22} \\
                 \end{array}
               \right)(K^{-1}_{11})^{T} Q^{T}, \,\,\,\,\   C^{\dag^{T}}QC^{\ddagger}_{A}=U^{-T}_{1}K^{-T}_{11}L^{-1}_{11}JL_{11}K^{-1}_{11}U^{T}_{1},$$  \,\,\,\,\,\,\,\,\,\  $$JAC^{\ddagger}_{A}=JL_{11}  K^{-T}_{11}U^{T}_{1}, \,\,\,\,\ C^{\ddagger}_{A}C^{\ddagger^{T}}_{A}= Q\left(
                 \begin{array}{c}
                   I \\
                   - L^{-1}_{22} L_{21} \\
                 \end{array}
               \right)K^{-1}_{11} K^{-T}_{11}\left(
                                               \begin{array}{cc}
                                                 I &  - L^{-1}_{22} L_{21} \\
                                               \end{array}
                                             \right)Q^{T}, \,\,\,\,\ CC^{\dag}= U_{1}\left(
                                                                            \begin{array}{cc}
                                                                              K_{11}K^{-1}_{11} & 0 \\
                                                                            \end{array}
                                                                          \right)U^{T}_{1}, $$
where $U=(U_{1},\, U_{2})$,  $H=[H_{1},\, H_{2}]$; $U_{1}$ and $H_{1}$ are respectively the submatrices of $U$ and $H$
obtained by taking the first $r$ columns. Putting all the above terms into \eqref{91} leads to
\begin{align}\label{91}
V&=\left(\|JAC^{\ddagger}_{A}\|^{2}_{2}+\|C^{\dag^{T}}QC^{\ddagger}_{A}\|^{2}_{2}\right)  (PQP)^{\dag^{2}}+\|C^{\ddagger}_{A}\|^{2}_{2}
\left(\|(PQP)^{\dag}A^{T}J\|^{2}_{2}+ C^{\ddagger}_{A}C^{\ddagger^{T}}_{A}  \right)+\|I-CC^{\dag}\|^{2}_{2}\|C^{\ddagger}_{A}C^{\dag^{T}}\|^{2}_{2}.\nonumber\\
\end{align}
\begin{align*}
V&= \left(\| JL_{11}  K^{-T}_{11}\|^{2}_{2}+\|K^{-T}_{11}L^{-1}_{11}JL_{11}K^{-1}_{11}\|^{2}_{2}\right)\left(Q\left(
                 \begin{array}{c}
                    0 \\
                   -( L^{T}_{22} L_{22})^{-1} \\
                 \end{array}
               \right)Q^{T}\right)^{2}\\
              &\quad +\left\|\left(
                 \begin{array}{c}
                   I \\
                   - L^{-1}_{22} L_{21} \\
                 \end{array}
               \right)K^{-1}_{11} \right\|^{2}_{2}  \left( \left\|\left(
                 \begin{array}{c}
                    0 \\
                    L^{-1}_{22} \\
                 \end{array}
               \right)H_{2}^{T} \right\|^{2}_{2}+Q\left(
                 \begin{array}{c}
                   I \\
                   - L^{-1}_{22} L_{21} \\
                 \end{array}
               \right)K^{-1}_{11} K^{-T}_{11}\left(
                                               \begin{array}{cc}
                                                 I &  - L^{-1}_{22} L_{21} \\
                                               \end{array}
                                             \right)Q^{T} \right)\\
                                            &\quad +\left \|I-U_{1}\left(\begin{array}{cc}
                                                                              K_{11}K^{-1}_{11} & 0 \\
                                                                            \end{array}
                                                                          \right) \right \|^{2}_{2}\left \|\left(
                 \begin{array}{c}
                    0 \\
                   - L^{-1}_{22} L_{22} \\
                 \end{array}
               \right)(K^{-1}_{11})^{T}  \right \|^{2}_{2}.
\end{align*}
}

\end{remark}
\begin{remark}\label{rem 3.4}
{\rm With the help of the expression of $\mathrm{d}(C^{\ddagger}_{A})$, we can find $\mathrm{d}x$, where
\begin{equation}\label{1001}
x=C^{\ddagger}_{A}h+(PQP)^{\dag}A^{T}Jg
\end{equation}
is the solution to the following ILSEP:
\begin{equation}\label{1000}
{\rm ILSEP:}\quad\mathop {\min }\limits_{\left\|g - Cx\right\|_2} (g-Ax)^{T}J(g-Ax),
\end{equation}
where $g \in \mathbb{R}^{m}$ and $h \in \mathbb{R}^{s}$.
 Specifically, differentiating both sides of  \eqref{1001} 
 , we have
\begin{align*}
\mathrm{d}x&= \mathrm{d}(C^{\ddagger}_{A}h+(PQP)^{\dag}A^{T}Jg).
\end{align*}
Thus, using  \eqref{3001}, we have
\begin{eqnarray*}
\mathrm{d}x&=& \mathrm{d}(C^{\ddagger}_{A}h+P(PQP)^{\dag}A^{T}Jg)=\mathrm{d}(C^{\ddagger}_{A})h+C^{\ddagger}_{A}\mathrm{d}h
+\mathrm{d}P(PQP)^{\dag}A^{T}Jg+P\mathrm{d}(PQP)^{\dag}A^{T}Jg \\
&&\quad\quad\quad\quad\quad \quad\quad\quad\quad\quad\quad+P(PQP)^{\dag}\mathrm{d}A^{T}Jg+P(PQP)^{\dag}A^{T}J\mathrm{d}g.
\end{eqnarray*}
Substituting  \eqref{3.5} into the above equation and using \eqref{2.10} lead to
\begin{eqnarray*}
\mathrm{d}x&=&[-C^{\ddagger}_{A}\mathrm{d}CC^{\ddagger}_{A}+C^{\ddagger}_{A}C^{\dag^{T}}\mathrm{d}C^{T}(I-CC^{\dag})
+(PQP)^{\dag}\mathrm{d}C^{T}C^{\dag^{T}}QC^{\ddagger}_{A}-(PQP)^{\dag} A^{T}J \mathrm{d}AC^{\ddagger}_{A}\nonumber  \\
&&-(PQP)^{\dag}\mathrm{d}A^{T}JAC^{\ddagger}_{A}]h+\mathrm{d}(I-C^{\dag}C)(PQP)^{\dag}A^{T}Jg+P[-(PQP)^{\dag}\mathrm{d}(PQP)(PQP)^{\dag}\nonumber\\
&&+(PQP)^{\dag}(PQP)^{\dag^{T}}\mathrm{d}(PQP)^{T}(I-(PQP)(PQP)^{\dag})\nonumber\\
&&+(I-(PQP)^{\dag}(PQP))\mathrm{d}(PQP)^{T}(PQP)^{\dag^{T}}(PQP)^{\dag}]A^{T}Jg\nonumber\\
&&+P(PQP)^{\dag}\mathrm{d}A^{T}Jg+P(PQP)^{\dag}A^{T}J\mathrm{d}g+C^{\ddagger}_{A}\mathrm{d}h,
\end{eqnarray*}
which together with  \eqref{3001}, \eqref{3002} and \eqref{3003} give
\begin{eqnarray*}
\mathrm{d}x&=&-C^{\ddagger}_{A}\mathrm{d}CC^{\ddagger}_{A}h+C^{\ddagger}_{A}C^{\dag^{T}}\mathrm{d}C^{T}(I-CC^{\dag})h
+(PQP)^{\dag}\mathrm{d}C^{T}C^{\dag^{T}}QC^{\ddagger}_{A}h-(PQP)^{\dag} A^{T}J \mathrm{d}AC^{\ddagger}_{A}h \nonumber  \\
&&-(PQP)^{\dag}\mathrm{d}A^{T}JAC^{\ddagger}_{A}h-C^{\dag}\mathrm{d}C(PQP)^{\dag}A^{T}Jg-(PQP)^{\dag}\mathrm{d}QP(PQP)^{\dag}A^{T}Jg\nonumber\\
&&-(PQP)^{\dag}Q\mathrm{d}P(PQP)^{\dag}A^{T}Jg+(PQP)^{\dag}PQP(PQP)^{\dag}\mathrm{d}P^{T}(I-(QP)(PQP)^{\dag})A^{T}Jg\nonumber\\
&&+(PQP)^{\dag}(PQP)^{\dag^{T}}\mathrm{d}Q^{T}P(I-(QP)(PQP)^{\dag})A^{T}Jg+P(PQP)^{\dag}\mathrm{d}A^{T}Jg+P(PQP)^{\dag}A^{T}J\mathrm{d}g+C^{\ddagger}_{A}\mathrm{d}h.\nonumber
\end{eqnarray*}
Noting \eqref{5003},  the above equation can be rewritten as
\begin{eqnarray*}
\mathrm{d}x&=&-C^{\ddagger}_{A}\mathrm{d}CC^{\ddagger}_{A}h+C^{\ddagger}_{A}C^{\dag^{T}}\mathrm{d}C^{T}(I-CC^{\dag})h
+(PQP)^{\dag}\mathrm{d}C^{T}C^{\dag^{T}}A^{T}JAC^{\ddagger}_{A}h-(PQP)^{\dag} A^{T}J \mathrm{d}AC^{\ddagger}_{A}h
\nonumber  \\
&&-(PQP)^{\dag}\mathrm{d}A^{T}JAC^{\ddagger}_{A}h-C^{\dag}\mathrm{d}C(PQP)^{\dag}A^{T}Jg+(PQP)^{\dag}\mathrm{d}A^{T}J(g-A(PQP)^{\dag}A^{T}Jg)\nonumber\\
&&-(PQP)^{\dag}A^{T}J\mathrm{d}A(PQP)^{\dag}A^{T}Jg+(PQP)^{\dag}QC^{\dagger}\mathrm{d}C(PQP)^{\dag}A^{T}Jg\nonumber\\
&&-(PQP)^{\dag}\mathrm{d}C^{T} C^{\dag^{T}}A^{T}J(g-A(PQP)^{\dag}A^{T}Jg)+(PQP)^{\dag}A^{T}J\mathrm{d}g+C^{\ddagger}_{A}\mathrm{d}h.\nonumber
\end{eqnarray*}

Further, by \eqref{3001} and \eqref{1001}, we have

\begin{eqnarray}
\mathrm{d}x&=&-C^{\ddagger}_{A}\mathrm{d}C(C^{\ddagger}_{A}h+(PQP)^{\dag}A^{T}Jg)+C^{\ddagger}_{A}C^{\dag^{T}}\mathrm{d}C^{T}(I-CC^{\dag})h
-(PQP)^{\dag} A^{T}J \mathrm{d}A(C^{\ddagger}_{A}h+(PQP)^{\dag}A^{T}Jg)
\nonumber  \\
&&-(PQP)^{\dag}\mathrm{d}C^{T} C^{\dag^{T}}A^{T}J(g-A(C^{\ddagger}_{A}h+(PQP)^{\dag}A^{T}Jg))\nonumber\\
&&+(PQP)^{\dag}\mathrm{d}A^{T}J(g-A(C^{\ddagger}_{A}h+(PQP)^{\dag}A^{T}Jg))+(PQP)^{\dag}A^{T}J\mathrm{d}g+C^{\ddagger}_{A}\mathrm{d}h\nonumber \quad \textrm{by \eqref{3001}} \\
&=&-C^{\ddagger}_{A}\mathrm{d}Cx+C^{\ddagger}_{A}C^{\dag^{T}}\mathrm{d}C^{T}\rho
-(PQP)^{\dag} A^{T}J \mathrm{d}Ax-(PQP)^{\dag}\mathrm{d}C^{T} C^{\dag^{T}}A^{T}Jr \nonumber  \\
&&+(PQP)^{\dag}\mathrm{d}A^{T}Jr+(PQP)^{\dag}A^{T}J\mathrm{d}g+C^{\ddagger}_{A}\mathrm{d}h  \quad \textrm{by \eqref{1001}} \label{4.5}
\end{eqnarray}

where $s=Jr=J(g-Ax)$, $\rho=(I-CC^{\dag})h$. By applying the `vec' operator on \eqref{4.5}, and noting \eqref{2.2} and \eqref{2.3}, we get
\begin{eqnarray*}
\mathrm{d}x &=&-(x^{T}\otimes(PQP)^{\dag}A^{T}J){\rm {vec}}(\mathrm{d}A)+(s^{T}\otimes(PQP)^{\dag}){\rm {vec}}(\mathrm{d} A^{T})-(x^{T}\otimes C^{\ddagger}_{A}){\rm {vec}}(\mathrm{d}C)\\
&&+(\rho^{T}  \otimes C^{\ddagger}_{A}C^{\dag^{T}}) {\rm {vec}}(\mathrm{d}C^{T})+(C^{\dag^{T}} A^{T}s)^{T}\otimes (PQP)^{\dag} ){\rm {vec}}(\mathrm{d}C^{T})+(PQP)^{\dag}\mathrm{d}g+C^{\ddagger}_{A}\mathrm{d}h \quad \textrm{ by \eqref{2.2}} \\
&=& [ -(x^{T}\otimes(PQP)^{\dag}A^{T}J)+(s^{T}\otimes(PQP)^{\dag})\Pi_{mn}]{\rm {vec}}(\mathrm{d}A)-[(x^{T}\otimes C^{\ddagger}_{A})-(\rho^{T}\otimes C^{\ddagger}_{A}C^{\dag^{T}})\Pi_{sn}\\
&&-((C^{\dag^{T}} A^{T}s)^{T}\otimes (PQP)^{\dag})\Pi_{sn}]{\rm {vec}}(\mathrm{d}C)  +(PQP)^{\dag}\mathrm{d}g+C^{\ddagger}_{A}\mathrm{d}h \quad \textrm{ by \eqref{2.3}}\\
&=& \big[-(x^{T}\otimes(PQP)^{\dag}A^{T}J)+(s^{T}\otimes(PQP)^{\dag})\Pi_{mn},  \,\  -(x^{T}\otimes C^{\ddagger}_{A})+(\rho^{T}\otimes C^{\ddagger}_{A}C^{\dag^{T}})\Pi_{sn}\\
&&+((C^{\dag^{T}} A^{T}s)^{T}\otimes (PQP)^{\dag})\Pi_{sn},  \,\
(PQP)^{\dag},  \,\    C^{\ddagger}_{A}\big]\begin{bmatrix}
{\rm {vec}}(\mathrm{d}A)\\
{\rm {vec}}(\mathrm{d}C)\\
\mathrm{d} g\\
\mathrm{d} h
\end{bmatrix}.
\end{eqnarray*}
By the above result, we can recover the condition numbers of ILSEP given in \cite{[5],[12]}.
Furthermore, note that $r=(g-A(C^{\ddagger}_{A}h +(PQP)^{\dag}A^{T}Jg))$. Thus, using the similar method, we can obtain $\mathrm{d}r$ and hence the condition numbers for residuals of ILSEP.
}
\end{remark}

\section{Statistical condition estimates}
\label{sec.6}
In this part, we focus on  estimating the normwise, mixed and componentwise condition numbers for the generalized inverse $C^{\ddagger}_{A}.$
\subsection{Estimating normwise condition number}
We use two algorithms to estimate the normwise condition number. The first one, outlined in Algorithm 1, is from  \cite{[18]} and has been applied to estimate the normwise condition number for matrix equations \cite{[25],[26]}, K-weighted pseudoinverse $K^{\dagger}_{L}$  \cite{[72]}, and indefinite least square problem \cite{[12]}. The second one, outlined in Algorithm 2, is based on the SSCE method \cite{[19]} and has been used for some least squares problems \cite{[12],[28],[72]}.
\begin{algorithm}[htbp]
\caption{Probabilistic condition estimator}\label{PCE}\small
$\textbf{Input:}$ $\epsilon$, $d$ ($d$ is the dimension of Krylov space and usually determined by the algorithm itself) and matrix $V$ in \eqref{91}.\\
$\textbf{Output:}$ Probabilistic spectral norm estimator of the normwise condition number \eqref{31}: $n^{\ddagger}_{p}(A,C).$
\begin{enumerate}
  \item Choose a starting vector $v_0$ uniformly and randomly from the unit $t$-sphere $S_{t-1}$ with $t=n^{2}$. 
  \item Compute the guaranteed lower bound $\alpha_1$ and the probabilistic upper bound $\alpha_2$ of $\|V\|_2$ by the probabilistic spectral norm estimator \cite{[18]} .
  \item Estimate the normwise condition number \eqref{31} by
   \begin{eqnarray*}
  n^{\ddagger}_{p}(A,C)=\frac{n_{p}(A, C)\|[A,\,\,C]\|_{F}}{\|C^{\ddagger}_{A}\|_{F}}\textrm{ with }  n_{p}(A, C)&=&\sqrt{\frac{\alpha_1+\alpha_2}{2}}.
  \end{eqnarray*}
 \end{enumerate}
\end{algorithm}
\begin{algorithm}[htbp]{
\caption{SSCE method for the normwise condition number of the generalized inverse $C^{\ddagger}_{A}$ }\label{SSCE}\small
  $\textbf{Input:}$ Sample size $k$ and matrix $V$ in \eqref{91}

  $\textbf{Output:}$ SSCE estimate of the normwise condition number of the generalized inverse $C^{\ddagger}_{A}:$ $n^{\ddagger}_{s}(A,C).$
\begin{enumerate}
 \item  Generate $k$ vectors ${{z}}_1,\cdots, {{z}}_k$ uniformly and randomly from the unit $n$-sphere $S_{n-1}$ and set $Z=[{{z}}_1,\cdots, {{z}}_k]$. 
  \item Orthonormalize these vectors using the QR facotization $[Z,\sim]=QR(Z)$.
  \item For $i=1,\cdots, k$, compute ${n^{\ddagger}_{i}}(K,L)$ by:
  \begin{align*}\label{4.1}
{n^{\ddagger}_{i}}(A,C)=  \frac{\sqrt{\sigma_{i}}\|[A,\,\,C]\|_{F}}{\|C^{\ddagger}_{A}\|_{F}},
\end{align*}
where
\begin{align*}
\sigma_{i}&=\left(\|JAC^{\ddagger}_{A}\|^{2}_{2}+\|C^{\dag^{T}}QC^{\ddagger}_{A}\|^{2}_{2}\right)  z^{T}_{i}(PQP)^{\dag^{2}}z_{i}+\|C^{\ddagger}_{A}\|^{2}_{2}
z^{T}_{i}\left((PQP)^{\dag}A^{T}A(PQP)^{\dag^{T}}+ C^{\ddagger}_{A}C^{\ddagger^{T}}_{A} \right)z_{i}\nonumber\\
&+\|I-CC^{\dag}\|^{2}_{2} z^{T}_{i} ( C^{\ddagger}_{A}C^{\ddagger^{T}}_{A}C^{\ddagger}_{A}C^{\ddagger^{T}}_{A})z_{i}.\nonumber
\end{align*}
  \item Approximate $\omega_k$ and $\omega_n$ by:
  \begin{eqnarray*}\label{4.3}
\omega_k\approx \sqrt{\frac{2}{\pi(k-\frac{1}{2})}}.
\end{eqnarray*}
   \item Estimate the normwise condition number \eqref{31} by:
   \begin{eqnarray*}\label{4.2}
n^{\ddagger}_{s} (A,C) = \frac{\omega_k}{\omega_n}\sqrt{\sum_{i=1}^k {{n^{\ddagger}}}^{2}_{i} (A,C)}.
\end{eqnarray*}
\end{enumerate}}
\end{algorithm}

\subsection{Estimating mixed and componentwise condition numbers}
 To  estimate  the  mixed  and  componentwise  condition  numbers,  we  need  the  following SSCE
method, which is from \cite{[19]} and has been applied to many problems (see e.g., \cite{[12],[24],[25],[26],[72]}).
\begin{algorithm}[htbp]
\caption{SSCE method for the mixed and componentwise condition numbers of the generalized inverse $C^{\ddagger}_{A}$ }\label{m&cSSCE}\small
  $\textbf{Input:}$ Sample size $k$ and matrix
  \begin{eqnarray*}
  \mathrm{d}\phi({\rm {vec}}(A), {\rm {vec}}(C))=\begin{bmatrix}W(A), W(C)
   \\\end{bmatrix}.
\end{eqnarray*}
  $\textbf{Output:}$ SSCE estimates of mixed and componentwise condition numbers of the generalized inverse $C^{\ddagger}_{A}:$ $m^{\ddagger}_{sce}(A,C),$ $c^{\ddagger}_{sce}(A,C).$
  \begin{enumerate}
 \item Let $t=mn+sn$. Generate $k$ vectors ${{z}}_1,\cdots, {{z}}_k$ uniformly and randomly from the unit $t$-sphere $S_{t-1}$ and set $Z=[{{z}}_1,\cdots, {{z}}_k]$.
  \item Orthonormalize these vectors using the QR facotization $[Z,\sim]=QR(Z)$.
  \item  Compute $u_i=\begin{bmatrix}W(A), W(C)
  \end{bmatrix}{z}_i$, and estimate the mixed and componentwise condition numbers in \eqref{312} and \eqref{313} by
  \begin{equation*}
 m^{\ddagger}_{sce}(A,C)=\frac{\|\kappa^{\dag}_{sce}\|_{\infty}}{\|{\rm {vec}}(C^{\ddagger}_{A})\|_{\infty}}, \quad c^{\ddagger}_{sce}(A,C)=\left\|\frac{\kappa^{\dag}_{sce}}{{\rm {vec}}(C^{\ddagger}_{A})}\right\|_{\infty},
  \end{equation*} where $\kappa^{\dag}_{sce}=\frac{\omega_k}{\omega_t}\left|\sum^{k}_{i=1}|u_i|^2\right|^{\frac{1}{2}}$, and the power and square root operation are performed on each entry of $u_i$, $i=1,\cdots, k$.
 \end{enumerate}
\end{algorithm}
\vspace{-6pt}
\newpage
\section{Numerical experiments}
\label{sec.7}
\vspace{-2pt}
Similarly to \cite{[29]},
we generate the matrices $A$ and $C$ as follows
\begin{equation*}
        A=H\begin{bmatrix}
         L_{11} & 0\\
         L_{21} & L_{22} \\
       \end{bmatrix}
Q^{T},\ C= U\begin{bmatrix}
                    K_{11} & 0\\
                     0 & 0 \\
                \end{bmatrix} Q^{T},
\end{equation*}
where $H$ is J-orthogonal, i.e., $H^{T} JH = J$, $Q$ is orthogonal, and $L_{22} \in \mathbb{R}^{(n\times s)\times(n\times s)}$ and $K_{11} \in \mathbb{R}^{(s\times s)}$ are lower triangular and nonsingular. In our experiment, we let $L_{11}$ be random matrix with full column rank and $L_{21}$ be arbitrary random matrix. $H$ is random J-orthogonal matrix with specified condition number and generated
via the method given in [16]. $Q$, $L_{22}$, and $K_{11}$ are generated by QR factorization of random matrices with
specified condition numbers and pre-assigned singular value distributions. In this case, the condition numbers of
$A$ and $C$ are $\kappa(A) = n^{l_{1}}$ and $\kappa(C) = s^{l_{2}}$, respectively.

In the following, using the above matrix pair, we first compare the condition numbers and their upper bounds, and then illustrate the reliability of Algorithms \ref{PCE}, \ref{SSCE}, and  \ref{m&cSSCE}.  All numerical experiments are performed in Matlab 2016a.

 In the specific experiments for comparing the condition numbers and their upper bounds, we set $p=50$, $q=30$,  $n=40$, and $s=20$. By varying the condition numbers of $A$ and $C$, we use $1000$ matrix pairs to test the performance.
The numerical results of the ratios defined by
\begin{eqnarray*}
  r_1 &=& n^{upper}(A,C)/n^{\ddagger}(A,C),\,\ r_2= m^{upper}(A,C) /m^{\ddagger}(A,C)\,\ \textrm{and} \,\ r_3= c^{upper}(A,C)/c^{\ddagger}(A,C)
\end{eqnarray*}
are presented in Table \ref{Table1}, which indicates that the upper bounds are quite reliable.
\begin{table}[!htb]
\centering
\caption{Comparisons of condition numbers and their upper bounds}\label{Table1}
\scriptsize
\begin{tabular}{c|c|llllllllllll}
  \hline
   & $\kappa(A)$ & \multicolumn{2}{c}{$n^{1}$} &\multicolumn{2}{c}{$n^{2}$} \\
\hline
          $\kappa(C)$   & $$  & mean & max &   mean  & max  \\
\hline
 $s^{0}$  &$r_{1}$& 1.6390e+00 &1.9197e+00 & 1.0531e+00&1.3291e+00\\
           &$r_{2}$& 1.9222e+00  & 9.0141e+00&1.7660e+00&8.2944e+00\\
           &$r_{3}$& 1.0223e+00& 3.0667e+00&1.0255e+00& 5.2405e+00 \\
\hline
 $s^{1}$ &$r_{1}$&  1.4372e+00& 1.5252e+00& 1.0308e+00& 1.0654e+00\\
           &$r_{2}$&  1.8006e+00 & 1.0951e+01& 2.0837e+00 &1.1066e+01\\
           &$r_{3}$& 1.1015e+00 & 3.8644e+00&1.2116e+00&7.9928e+00 \\
\hline
 $s^{2}$ &$r_{1}$&  1.4156e+00& 1.4225e+00& 1.0365e+00& 1.0531e+00\\
           &$r_{2}$& 1.9367e+00& 1.2151e+01& 2.0581e+00& 1.7458e+01\\
           &$r_{3}$& 1.2405e+00& 4.6376e+00& 1.3910e+00& 1.4761e+01\\
\hline
 $s^{3}$ &$r_{1}$& 1.4136e+00& 1.4149e+00&  1.0396e+00& 1.0596e+00\\
           &$r_{2}$& 2.0345e+00& 1.5120e+01& 2.1578e+00& 3.1794e+01\\
           &$r_{3}$&1.3526e+00 & 9.8179e+00&1.4360e+00& 1.4452e+01\\
\hline
&$\kappa(A)$  & \multicolumn{2}{c}{$n^{3}$} &\multicolumn{2}{c}{$n^{4}$}\\
\hline
 $s^{0}$ &$r_{1}$& 1.0068e+00  & 1.4799e+00&  1.0015e+00&1.0037e+00\\
           &$r_{2}$& 1.6997e+00& 1.0488e+01  &1.6313e+00&1.3419e+01\\
           &$r_{3}$&1.0891e+00& 4.1369e+00 &1.0881e+00& 4.0696e+00 \\
\hline
 $s^{1}$ &$r_{1}$&  1.0013e+00&  1.0208e+00&  1.0043e+00&  1.0145e+00\\
           &$r_{2}$& 2.4950e+00 &  8.6450e+01& 2.3762e+00 &6.8388e+01\\
           &$r_{3}$&1.3941e+00 & 2.2704e+01&1.4708e+00&2.8664e+01\\
\hline
 $s^{2}$ &$r_{1}$&  1.0003e+00& 1.0011e+00&1.0062e+00&1.0092e+00\\
           &$r_{2}$& 2.4379e+00& 4.9124e+01& 2.6558e+00& 8.6109e+01\\
           &$r_{3}$&1.5545e+00 &1.6603e+01&1.7349e+00&4.3272e+01\\
\hline
 $s^{3}$ &$r_{1}$& 1.0002e+00& 1.0013e+00&1.0007e+00&1.0054e+00\\
           &$r_{2}$& 2.5608e+00& 8.7733e+01& 2.2141e+00& 3.8474e+01\\
           &$r_{3}$& 1.6588e+00 &4.0571e+01&1.6631e+00&5.7799e+01\\
\hline
\end{tabular}
\end{table}

 For Algorithm \ref{PCE}, we choose the parameters: $\delta= 0.01$ and $\epsilon=0.001$.
 For Algorithms \ref{SSCE} and \ref{m&cSSCE}, we set $k=3$. By varying the condition numbers of $A$ and $C$, we have the numerical results on
 the ratios defined as follows:
\begin{eqnarray*}
  r_p &=& n^{\ddagger}_{p}(A,C)/n^{\ddagger}(A,C),\;r_s=n^{\ddagger}_{s}(A,C)/n^{\ddagger}(A,C), \\
    r_m &=&  m^{\ddagger}_{sce}(A,C)/m^{\ddagger}(A,C),\ r_c= c^{\ddagger}_{sce}(A,C)/c^{\ddagger}(A,C).
\end{eqnarray*}
We present these numerical results and CPU time in Figures 1--2. The time ratios are defined by
\begin{eqnarray*}
t_{p}:= \frac{t_{1}}{t}, \quad t_{s}:= \frac{t_{2}}{t}, \quad
t_{m}&:=&\frac{t_{3}}{t} , \quad   t_{c}:=\frac{t_{4}}{t},
\end{eqnarray*}
where $t$ is the CPU time of computing the generalized inverse $C^{\ddagger}_{A}$  by  GHQR decomposition \cite{[3]} and $t_{1}$, $t_{2}$, $t_{3}$ and $t_{4}$  are the CPU time of Algorithm 1, 2 and 3.
These results suggest that these three algorithms are very effective and reliable in estimating condition numbers.

\end{document}